\documentclass[11pt]{amsart}
\usepackage[cm]{fullpage}
\usepackage{amssymb,amsmath,amsthm,amscd,mathrsfs,graphicx}
\usepackage[cmtip,all]{xy}
\usepackage{pb-diagram}
\usepackage{xcolor}
\usepackage{tikz}

\numberwithin{equation}{section}

\newtheorem{theorem}{Theorem}[section]
\newtheorem{proposition}{Proposition}[section]
\newtheorem{lemma}{Lemma}[section]

\newcommand{\ov}[1]{\overline{#1}}

\newcommand{\ve}{\varepsilon}

\theoremstyle{definition}

\theoremstyle{remark}

\begin{document}
\bibliographystyle{amsplain}

\author[A. Chau]{Albert Chau}
\address{Department of Mathematics, The University of British Columbia, 1984 Mathematics Road, Vancouver, B.C.,  Canada V6T 1Z2.  Email: chau@math.ubc.ca. } 
\author[B. Weinkove]{Ben Weinkove}
\address{Department of Mathematics, Northwestern University, 2033 Sheridan Road, Evanston, IL 60208, USA.  Email: weinkove@math.northwestern.edu.}

\title{The Stefan problem and concavity}

\maketitle

\vspace{-20pt}

\begin{abstract}
We construct examples for the one-phase Stefan problem which show that $\alpha$-concavity of the solution is in general not preserved in time, for $0 \le \alpha <1/2$.  In particular, this shows that, in contrast to the case of the heat equation for a fixed convex domain, log concavity is not preserved for solutions of the Stefan problem.
\end{abstract}

\section{Introduction}
 
 The one-phase Stefan problem is a free boundary problem used to model phase transitions in matter where the phase boundary moves with time.  The initial data is given by a domain $\Omega_0$ in $\mathbb{R}^n$ and a function $u_0$ defined on its closure $\ov{\Omega}_0$ which vanishes on the boundary $\partial \Omega_0$ and is positive in the interior $\Omega_0$.  For positive time $t$, the solution of the Stefan problem is a family of domains $\Omega_t$ and a function $u$, positive on $\Omega_t$, such that 
\begin{equation}\label{laplace}
\begin{split}
u_t &= \Delta u \,\,\,\text{in}\,\,\, \Omega_t\\
u&=0  \,\,\,\text{on}\,\,\, \partial \Omega_t, 
\end{split}
\end{equation}
and that the \emph{Stefan boundary condition} holds, which states that $\partial \Omega_t$ moves in the direction of the outward normal with speed $|\nabla u|$.  This can be stated as follows:  if $X(t)$ is a  path in $\partial \Omega_t \subset \mathbb{R}^n$ whose derivative is normal to $\partial \Omega_t$ at $X(t)$ then
\begin{equation} \label{sbc}
\dot{X}(t) = - \nabla u(X(t), t).
\end{equation}
The study of the Stefan problem has a long history, and we refer the reader to \cite{C, CF, F1, FK, H, Kim, K, KN,  M, O, R1, R2} and the references therein for the basic existence and uniqueness results.

In a classic paper, Brascamp-Lieb \cite{BL} showed that log concavity is preserved along the heat equation on convex domains (see also \cite{Bo, CW1, CW2, CMS, DK, IS1, IS2, IS3} for some related results).   Daskalopoulos-Hamilton-Lee \cite{DHL} showed that root concavity ($1/2$-concavity in the terminology below) is preserved for the porous medium equation, a degenerate diffusion equation.  It is natural then to ask what if any concavity conditions are preserved for the Stefan problem \cite{DL}.  

Our main result is a negative one.  We show that $\alpha$-concavity of the initial data, for $\alpha \in [0,1/2)$, is \emph{not} in general preserved in time for the Stefan problem.  Log concavity corresponds to $\alpha=0$ and so our result implies in particular that log concavity is not preserved.

More precisely, we will now define what we mean by a solution of the Stefan problem in the special case that it is of interest to us.  Fix $k \ge 2$. 
 Let  $\Omega_0$ be a bounded domain  and let 
 $u_0$ be a function in $C^k(\ov{\Omega}_0)$ which vanishes on $\partial \Omega_0$, is strictly positive on $\Omega_0$ and whose derivative $\nabla u_0$ is nowhere vanishing on $\partial \Omega_0$.   We define a nondegenerate $C^k$ solution of the (one-phase) Stefan problem starting with the initial data $(\Omega_0, u_0)$  on the time interval $[0,T]$ to be a relatively open set $\Omega \subset \mathbb{R}^n \times [0,T]$  and a function $u \in C^k(\ov{\Omega})$ satisfying the following conditions.  For each $t \in [0,T]$, the set $\Omega_t := \Omega \cap (\mathbb{R}^n \times \{t \})$ is a bounded domain in $\mathbb{R}^n \times \{t \} \cong \mathbb{R}^n$ with $t=0$ corresponding to the initial domain $\Omega_0$.  
 The function $u|_{\ov{\Omega}_0}$ is equal to the initial data $u_0$.  The function $u$ is positive on $\Omega$ and vanishes on $\partial \Omega$, and $u$ satisfies \eqref{laplace} and (\ref{sbc}).  Moreover, $u$ satisfies the nondegeneracy condition that its spatial gradient $\nabla u$ does not vanish anywhere on $\partial \Omega$.
  
 Some remarks are in order:

\begin{enumerate} 
\item In the above we use the parabolic definition of $C^k$, so that $u \in C^k$ means that $u$ has $k$ derivatives in the spatial directions and $k/2$ derivatives in the time direction. 
 \item When we refer to  the boundary $\partial \Omega = \ov{\Omega} \setminus \Omega$, we are using the subspace topology on $\mathbb{R}^n \times [0,T]$.  In particular $\Omega$ includes $\Omega_0$ and $\Omega_T$ while $\partial \Omega$ does not.  By the nondegeneracy condition on $u$, the boundary $\partial \Omega$ inherits regularity from $u$.

\item  Our definition of a solution to the Stefan problem is rather restrictive since it insists that  $u$ be $C^k$ at time $t=0$ in both space and time directions, up to the boundary.  This imposes compatibility conditions on $u_0$ which are described in Section \ref{prelim} below.
\end{enumerate}

We now define $\alpha$-concavity in our setting.  Let $W\subset \mathbb{R}^n$ be a bounded domain and let $v \in C^2(\ov{W})$ be positive on $W$ and vanishing on $\partial W$. 
For $\alpha>0$, we say that $v$ is  $\alpha$-concave  on $W$ if $D^2v^{\alpha} \le 0$ on $W$.  We say that $v$ is $0$-concave if $D^2 \log v \le 0$ on $W$.  Equivalently, $\alpha$-concavity corresponds to the nonpositivity of the matrix with $(i,j)$th entry
 $$(v D_iD_j v - (1-\alpha) D_iv D_j v)$$
 on $W$.  
 
Our main theorem gives a family of examples in $\mathbb{R}^2$ for the Stefan problem which break the $\alpha$-concavity for positive time.

\begin{theorem} \label{maintheorem}
Given $k \ge 2$ and $\alpha \in [0,1/2)$, there exist a bounded convex set $\Omega_0 \subset \mathbb{R}^2$ with smooth boundary and $u_0 \in C^{\infty}(\ov{\Omega}_0)$ which is strictly positive on $\Omega_0$ and vanishes on $\partial \Omega_0$ with the following properties:
\begin{enumerate}
\item[(i)] $u_0$ is $\alpha$-concave on $\Omega_0$.
\item[(ii)] $\nabla u_0$ does not vanish at any point of $\partial \Omega_0$.
\item[(iii)] There exists $T>0$ and a unique nondegenerate $C^k$ solution $(\Omega,u)$   of the Stefan problem starting at the initial data $(\Omega_0, u_0)$ on the time interval $[0,T]$ such that: 
$$\Omega_t \textrm{ is not convex for any $t \in (0,T]$}$$  and $$u|_{\Omega_t} \textrm{ is not $\alpha$-concave for any $t\in (0,T]$}.$$
\end{enumerate}
\end{theorem}

Our result in particular implies that additional assumptions are needed for Theorem 1.1 of \cite{DL}.
 
The outline of the paper is as follows.  In Section \ref{prelim} we give an overview of the compatibility conditions required for short time existence of $C^k$ solutions to the Stefan problem (in the sense described above).  We also give some elementary results about concave functions, including their short proofs. Section \ref{sectionconstruct} is the main part of the paper, which gives the construction of $\Omega_0$ and $u_0$.  The proof of Theorem \ref{maintheorem} is completed in Section \ref{sectionlast}.

\section{Preliminaries} \label{prelim}

\subsection{Short time existence result and compatibility conditions}  The Stefan boundary condition (\ref{sbc}) together with \eqref{laplace} induces compatibility conditions  for $u_0$ on the boundary $\partial \Omega_t$.   These arise from differentiating with respect to $t$ the equation
\begin{equation}\label{comp0}
u(X(t), t)=0
\end{equation}
along a path $X(t) \in \partial \Omega_t$ whose derivative is normal to the boundary, while using  \eqref{sbc} and \eqref{laplace} to replace time derivatives with spatial derivatives.   We illistrate this by deriving the first two compatibility conditions in detail as follows.

 Differentiating \eqref{comp0} once in time gives $u_t+\nabla u \cdot \dot{X} =0$, then using  \eqref{laplace} and  \eqref{sbc} gives the \emph{first compatibility condition}  for $u_0$:

 \begin{equation} \label{comp1}
 \Delta u_0 - |\nabla u_0|^2 =0  \,\,\,\text{on}\,\,\, \partial \Omega_0.
 \end{equation}
Differentiating (\ref{comp0}) once more in time gives $\Delta u_t+ \nabla\Delta u \cdot \dot{X}  -2\nabla u_t \nabla u-2 u_i u_{ij} \dot{X}_j =0 $, and using  \eqref{laplace} and  \eqref{sbc} we obtain the  \emph{second compatibility condition}  for $u_0$:

 \begin{equation} \label{comp2}
 \Delta^2 u_0 -3\nabla\Delta u_0 \cdot \nabla u_0 +2 \sum_{i,j} (u_0)_{ij}(u_0)_i (u_0)_j =0  \,\,\,\text{on}\,\,\, \partial \Omega_0.
 \end{equation}

In general, we see that the \emph{$k$th compatibility condition} for the initial condition for $u_0$ can be written as

 \begin{equation} \label{compk}
 \Delta^k (u_0) +  F_k (u_0)=0  \,\,\,\text{on}\,\,\, \partial \Omega_0,
 \end{equation}
where $F_k$ is  a differential operator of degree at most $2k-1$ and is obtained as above, namely by differentating \eqref{comp0} $k$ times in $t$, then using  \eqref{laplace} and  \eqref{sbc} to replace time derivatives of $u_0$ with spatial derivatives.

A result of Hanzawa \cite{H} (see also \cite{M}) states that solutions to the Stefan problem exist on a small time interval $[0,T]$ as long as $u_0$ satisfies compatibility conditions.  For our purposes we may assume that we are given the initial data of a smooth function $u_0 \in C^{\infty}(\ov{\Omega}_0)$ where $\Omega_0$ is a bounded domain whose boundary $\partial \Omega_0$ is  smooth and has only one component.  The function $u_0$ is strictly positive on $\Omega_0$, vanishes on  $\partial \Omega_0$ and its derivative $\nabla u_0$ is nowhere vanishing on $\partial \Omega_0$.  In this setting, Hanzawa's result can be stated as follows.

\begin{theorem} \label{thmste}
Fix $k \ge 2$.  Then there exists $N=N(k)$ such that if $u_0$ satisfies the first $N$ compatibility conditions then there exists $T>0$ and a unique nondegenerate $C^k$ solution of the Stefan problem $(\Omega, u)$ on the time interval $[0,T]$ starting with this initial data.
\end{theorem}

The constant $N=N(k)$ is given explicitly in \cite{H} and is not optimal, but here we are not concerned with the question of optimal regularity.

\subsection{Two elementary lemmas about concave functions}

In this section we recall two  known, elementary results which will be needed in the sequel.  Let $r$ denote the distance from the origin in $\mathbb{R}^2$.

\begin{lemma}\label{l0}
Let $q(r)$ be a $C^2$ radial function on $\mathbb{R}^2$.  Then $q$ is strongly concave on the set 
$$S=\{ (x,y) \in \mathbb{R}^2 \ | \  r\neq 0, \ q'(r)<0, \ \textrm{and }  q''(r)<0 \}.$$
\end{lemma}
\begin{proof} This is a straightforward computation, using the fact that $r_x = x/r$ and $r_y = y/r$.  At a point in $S$ we have
\[
\begin{split}
& q_x = \frac{x}{r} q', \quad q_{xx} = \frac{x^2}{r^2} q'' + \frac{1}{r} q' - \frac{x^2}{r^3} q' = \frac{x^2}{r^2} q'' + \frac{y^2}{r^3}q '<0 \\
& q_y = \frac{y}{r} q', \quad q_{yy} = \frac{y^2}{r^2} q'' + \frac{1}{r} q' - \frac{y^2}{r^3} q' = \frac{y^2}{r^2} q'' + \frac{x^2}{r^3} q'<0\\
& q_{xy} = \frac{xy}{r^2} q'' - \frac{xy}{r^3} q', \quad q_{xx}q_{yy} - q_{xy}^2 = \frac{1}{r} q' q'' >0,
\end{split}
\]
as required.
\end{proof}

Next we have the following elementary lemma about extending concave functions \cite{G}.

\begin{lemma}\label{Ghomi}
Let $W$ be an open bounded convex set in $\mathbb{R}^n$ and let $f$ be a smooth real-valued function defined on the set
$$W^{\delta} = \{ p \in W \ | \ \emph{dist}(p, \partial W) < \delta \},$$
for some $\delta>0$.  Assume that there exists a constant $c \in [-\infty, \infty)$ such that $f(p)>c$ on $W^{\delta}$ and $f(p) \rightarrow c$ as $p$ tends to any point in $\partial W$.  Also assume that $D^2 f <0$ and $Df \neq 0$ on $W^{\delta}$.   Then there exists a smooth concave function $F: W \rightarrow (c, \infty)$ which coincides with $f$ on $W^{\delta'}$ for some $0<\delta'<\delta$.  

Moreover, if $W$ is a ball in $\mathbb{R}^n$ centered at a point $P$ and $f$ a  function of the distance $r$ from $P$, then $F$ can also be taken to be a function of $r$.
\end{lemma}
\begin{proof} The proof is essentially contained in \cite{G}, but we give the argument here for the sake of completeness.
First note that the level sets $\{ f = a \}$ for constants $a$ close to (and strictly larger than) $c$ are smooth convex hypersurfaces contained in $W^{\delta}$.  If $c=-\infty$ then ``$a$ close to $c$'' means that $a$ is sufficiently negative.

Fix now such an $a>c$.  Define a function $\tilde{f}: W \rightarrow \mathbb{R}$ by
$$\tilde{f} = \left\{ \begin{array}{ll} f,  \quad & \textrm{on } \{ f \le a \} \\ a, \quad  & \textrm{otherwise.} \end{array} \right.$$
Then the function $\tilde{f}$ is locally concave on $W$ and hence concave on $W$ (away from $\{ f=a \}$ it satisfies $D^2 \tilde{f} \le 0$ and near $\{ f= a\}$ it is the minimum of two concave functions).

For small $\ve>0$ let $\tilde{f}_{\ve} = \tilde{f} * \eta_{\ve}$ be the convolution of $\tilde{f}$ by a mollifier $\eta_{\ve}$ given by $\eta_{\ve}(p) = \ve^{-n} \eta (\ve^{-1}p)$ for $\eta$ a smooth nonnegative  function supported in the unit ball $B$ with $\int_B \eta dx=1$.  Then $\tilde{f}_{\ve}$ is smooth and concave on its domain of definition $W_{\ve} := \{ p \in W \ | \ \textrm{dist}(p, \partial W) > \ve \}.$

Now fix $a'$ with $a>a'>c$.  Then $\tilde{f}$ is smooth and concave on $\{ f < a' \}$.  Choose a smooth bump function $\phi$ which is equal to $0$ on $\{ f \ge a' \}$ and is equal to $1$ on $W^{\delta'}$ for some small $\delta'>0$.

Define a smooth function $F: W \rightarrow \mathbb{R}$ by 
$$F = (1-\phi) \tilde{f}_{\ve} + \phi \tilde{f}.$$
Then $F$ agrees with the concave function $\tilde{f}$ and hence $f$ on $W^{\delta'}$ and agrees with the concave function $\tilde{f}_{\ve}$ on $\{ f \ge a'\}$.  Here we choose $\ve>0$ small enough so that $\tilde{f}_{\ve}$ is defined on the complement of $W^{\delta'}$.  It remains to check that $F$ is concave on the compact set $K=\{ f \le a' \} \setminus W^{\delta'}$.
 Write $F = \tilde{f}_{\ve} + \phi (\tilde{f} - \tilde{f}_{\ve})$.  Then
$$F_{ij} = (\tilde{f}_{\ve})_{ij} + \phi_{ij} (\tilde{f} - \tilde{f}_{\ve}) + \phi_i (\tilde{f}-\tilde{f}_{\ve})_j + \phi_j (\tilde{f}-\tilde{f}_{\ve})_i + \phi (\tilde{f} - \tilde{f}_{\ve})_{ij}.$$
The result now follows from the fact that $\tilde{f}_{\ve} \rightarrow f$ uniformly $C^2$ on $K$ as $\ve \rightarrow 0$.  Indeed,  since $D^2f<0$ on the compact set $K$, it is uniformly strongly concave on $K$, so we can choose $\ve>0$ small enough that that the matrix $(\tilde{f}_{\ve})_{ij}$ is uniformly negative definite.  But since $\tilde{f}_{\ve} \rightarrow f$ uniformly $C^2$ on $K$ as $\ve \rightarrow 0$, all the other terms tend to zero, so  for $\ve>0$ sufficiently small we get $(F_{ij}) <0$ on $K$.

Finally, in the case when $W$ is a ball in $\mathbb{R}^n$ centered at a point $P$ and $f$ is a function of the distance $r$ from $P$, we define $\tilde{f}_{\ve}(r)$ to be the convolution of $\tilde{f}(r)$ with a standard mollifier as a function of $r$.  We also choose the bump function $\phi$ to be a function of $r$, and it follows that $F$ is a function of $r$. 
\end{proof}

\section{Construction of the initial data} \label{sectionconstruct}

 Our starting point is a radial function defined on a disk.  Fix an integer $N \ge 1$ and some $\alpha \in [0,1/2)$ throughout the section.
Let $D$ be the open disk in $\mathbb{R}^2$ of radius 2 centered at the point $(0,2)$, and let $r$ denote the distance from $(0,2)$.

\begin{proposition}\label{prop1}  
  There exists a smooth function $U$ on the closed disk $\ov{D}$ satisfying
\begin{enumerate}
\item[(a)] $U$ is positive on $D$, vanishes on $\partial D$, and $\nabla U$ is nowhere zero on $\partial D$.
\item[(b)] $U$ is $\alpha$-concave on $D$.
\item[(c)] $U$ satisfies the compatibility conditions $(\ref{compk})$ on $\partial D$, for $k=1, \ldots, N$.
\end{enumerate}
Moreover, $U$ is a function of $r$.
\end{proposition}

\begin{proof}
Such a function $U(x, y)$ is easily constructed as follows.  Writing $r$ for the distance from $(0,2)$, let $q: [0,2] \rightarrow \mathbb{R}$ be 
a smooth function satisfying
\begin{equation}
q(r) >0, \ 0\le r <2, \ q(2)=0
\end{equation}
and
\begin{equation} \label{parta}
q'(2)<0, \ q''(2)+ \frac{1}{2} q'(2)>0.
\end{equation}
Then for $\alpha \in (0,1/2)$ we see that for $r$ close to $2$ we have
$$(q^{\alpha})' =  \alpha q^{\alpha -1} q' <0,$$
and
$$(q^{\alpha})'' =  \alpha q^{\alpha-2} ( (\alpha-1) (q')^2 + q q'')<0.$$
For the case $\alpha=0$ replace $q^{\alpha}$ by $\log q$ and the same holds.
This implies that if we define $U(x,y) = q(r)$ for $r$ close to $2$ then $U$ is strongly $\alpha$-concave there by Lemma \ref{l0}. To satisfy $\Delta U= |\nabla U|^2$ on $\partial D$ we note that by (\ref{parta}), the quantities $\Delta U$ and $| \nabla U|^2$ are strictly positive on $\partial D$ and hence we can scale $U$ to ensure that the first compatibility condition holds.

We can then recursively prescribe $q^{(2k)}(2)$ for $k=2,3,\ldots, N$ so that the $k$th compatibility condition (\ref{compk}) for $U$ holds up to order $N$.  This does not affect the positivity or $\alpha$-concavity of $U$ near $\partial D$.  We can then extend $U$ by Lemma \ref{Ghomi} to a positive $\alpha$-concave function inside $D$.  
\end{proof}

We now define a new convex domain $\Omega_0$, obtained by modifying the disk $D$.  The part of the boundary $\partial D$ below the line $y=2$ can be written as a graph 
$y=G(x)$ for
\begin{equation} \label{defnG}
G(x) = 2 - \sqrt{4-x^2}, \quad -2<x<2.
\end{equation}
Let $\delta \in (0,1/4)$ be a small positive constant depending only on $\alpha$, to be determined later.   
Define a new function $g:(-2,2) \rightarrow \mathbb{R}$ by modifying $G$ as follows:
\begin{equation} \label{defng}
g(x) = \left\{ \begin{array}{ll}  1/20, &  \qquad x \in [-\delta, \delta ] \\ G(x), & \qquad x \in (-2, -1/2] \cup [1/2, 2) \end{array} \right.
\end{equation}
and extend $g$ to be a smooth function on $(-2,2)$ so that $g''>0$ on the remaining intervals $(-1/2, -\delta)$ and $(\delta, 1/2)$.  Note that $g$ is a convex function.  We now define our convex domain to be
$$\Omega_0 = D \cap \{ (x,y) \in (-2,2) \times \mathbb{R} \ | \ y > g(x) \}.$$
Note that $\partial \Omega_0$ is smooth, and contains a line segment $[-\delta, \delta] \times\{1/20 \}$.

The main result of this paper is the following construction:

\begin{theorem} \label{thmconstruction}
There is a smooth function $v$ on $\ov{\Omega}_0$ satisfying the following conditions:
\begin{enumerate}
\item[(a)] $v$ is positive on $\Omega_0$, vanishes on $\partial \Omega_0$.
\item[(b)] $v$ is $\alpha$-concave on $\Omega_0$.
\item[(c)] $\nabla v$ does not vanish at any point of $\partial \Omega_0$.
\item[(d)] $v$ satisfies the compatibility conditions $(\ref{compk})$ on $\partial \Omega_0$, for $k=1, \ldots, N$.
\item[(e)] The map
$$x\mapsto v_y(x, 1/20), \qquad \textrm{for } x \in [-\delta, \delta],$$
is positive and strongly convex (namely $(v_{y})_{xx} >0$). 
\end{enumerate}
\end{theorem}
\begin{proof}  We obtain $v$ by modifying the function $U$ constructed  in Proposition \ref{prop1}, which we may assume for convenience is defined as a smooth function on 
 all of $\mathbb{R}^2$.   We may write $U(x,y)$ on $[-1,1] \times [-1,1]$ as
 $$U(x,y) = \sum_{\ell = 1}^{2N} (y-G(x))^{\ell} E_{\ell}(x)+  (y-G(x))^{2N+1} R(x, y)$$
 for smooth functions $ E_1, \ldots, E_{2N}, R$, where we recall that $G$ is defined by (\ref{defnG}).  Indeed, this follows by considering the new variable $ \tilde{y}=y-G(x)$, and considering for each $x\in [-1,1]$ the Taylor expansion of the function $U(x,  \tilde{y})$ in powers of $ \tilde{y}$ around $ \tilde{y}=0$ and using the Taylor remainder theorem.

 We now define a function $v(x,y)$ on $[-1,1]\times[-1,1]$ by
 \begin{equation} \label{defnv}
 v(x,y) =\sum_{\ell=1}^{2N} (y-g(x))^{\ell} e_{\ell}(x)+  (y-g(x))^{2N+1} R(x, y)
 \end{equation}
 where $g(x)$ is as in (\ref{defng}) and the smooth functions $e_1, \ldots, e_{2N}$ are to be determined.

 For simplicity of notation, in what follows we will denote $E_1, E_2$ by $F, H$ and $e_1, e_2$ by $f, h$ respectively.  Note that since $\nabla U$ is nonvanishing on $\partial D$, there exists a constant $c>0$ such that
 \begin{equation} \label{Fc}
 F(x) \ge c >0.
 \end{equation}
 We define $f$ as follows.  First choose $\delta \in (0,1/4)$ so that the solution $\psi(x)$ of the ODE problem
 $$\psi \psi'' - \frac{2}{3} (1-2\alpha) (\psi')^2 =-1, \quad \textrm{on } [-2\delta, 2\delta],$$
 $$\psi(0)=1, \psi'(0) = \sqrt{2/(1-2\alpha)}$$
 satisfies $\psi'' >0$ and $\psi>0$ on $[-2\delta, 2\delta]$.  Then we define $f$ to be a smooth function on $[-1,1]$ such that 
 $$f(x) =  \left\{ \begin{array}{ll}  \psi(x), & \quad  \textrm{for }  -2\delta \le x \le 2\delta, \\
 F(x), & \quad \textrm{for }  x \in [-1, -1/2] \cup [1/2, 1] \end{array} \right. $$
 and 
 \begin{equation} \label{lbf}
 f(x) \ge c>0 \quad \textrm{on } [-1,1],
 \end{equation} where we assume without loss of generality that $c$ is the same constant as in (\ref{Fc}).
 Observe that from the definition of $f$ and $g$, we have $v_y(x, 1/20)=\psi(x)$ for $x \in [-\delta, \delta]$ and hence  $v(x,y)$ will satisfy (e).
 
 We are interested only in the behavior of $v$ close to the boundary $y=g(x)$ and so in what follows we can neglect some higher order terms of $|y-g|$.  Write
 $$v = (y-g) f + (y-g)^2h + O(|y-g|^3)$$
Differentiating this we obtain
 \begin{equation} \label{vcomp}
 \begin{split} v_y = {} & f + 2(y-g)h  + O(|y-g|^2) \\
 v_x = {} & -g'f + (y-g)f' - 2(y-g)g' h + (y-g)^2 h' \\ {} & + (y-g)^2 O(|g'|)+ O(|y-g|^3) \\
v_{yy} = {} &    2h+ O(|y-g|) \\
v_{xy} =  {} & f'   -2g'h +2(y-g)h' + (y-g) O(|g'|)+  O(|y-g|^2) \\
v_{xx} = {} & -g''f - 2g'f'+ 2(g')^2 h +(y-g)f''  +   (y-g)O(|g'|, |g''|)  + O(|y-g|^2).
\end{split}
\end{equation}
We will now choose the function $h(x)$ so that $v$ satisfies the first compatibility condition $\Delta v = |\nabla v|^2$ on the boundary $\{ y=g(x) \}$.  Observe that on $(x,g(x))$ for $x \in [-1,1]$,
$$\Delta v =  2h -g''f - 2g'f' + 2(g')^2h$$
and
$$|\nabla v|^2 =  v_x^2+v_y^2 = (g')^2 f^2 + f^2,$$
so that we require 
\begin{equation} \label{hxgx}
h(x) = \frac{(g')^2f^2+ f^2 +g''f+2g'f'}{2(1+(g')^2)}.
\end{equation}
We define $h(x)$ by (\ref{hxgx}).   It follows that $v(x,y)$  satisfies the first compatibility condition $\Delta v = |\nabla v|^2$  on $\partial\Omega_0 \bigcap  [-1, 1]\times[-1,1]$.
Moreover, since our function $U$ constructed in Proposition \ref{prop1} satisfies the first compatibility condition $\Delta U = |\nabla U|^2$ on the boundary, we have
$$H(x) = \frac{(G')^2F^2+ F^2  +G''F+2G'F'}{2(1+(G')^2)},$$
for $x\in [-1,1]$.  Since $f=F$ and $g=G$ whenever $x$ is in the complement of $[-1/2, 1/2]$, it follows that $h=H$ in the complement of $[-1/2,1/2]$.
 
 We can now similarly define $e_3, e_4, \ldots, e_{2N}$ so that  $v(x,y)$ satisfies  the $k$th compatibility conditions  for $k=2, 3, \ldots, N$.
  Indeed applying the operator $\Delta^2$ to (\ref{defnv}) and evaluating on $y=g(x)$, the second compatibility condition takes the form
 \begin{equation}\label{e4}
 \begin{split}
e_4(x) = {} &  \textrm{smooth expression in terms of $e_3(x)$, $g(x)$, $f(x)$, $h(x)$} \\ {} & \textrm{ and their derivatives.}
 \end{split}
 \end{equation}
We can simply define $e_3(x) = E_3(x)$.
Define $e_4(x)$ by the formula (\ref{e4}) so that $v(x,y)$ satisfies the second compatibility condition.  Then since $U$ satisfies the second compatibility condition, and $f=F$, $g=G$, $h=H$, $e_3=E_3$ when $x$ is in the complement of $[-1/2, 1/2]$ it follows that $e_4=E_4$ in the complement of $[-1/2,1/2]$.

Continuing inductively, we define $e_5, \ldots, e_{2N}$ so that the $k$th compatibility conditions are satisfied  for $k=2, 3, \ldots, N$. Moreover  when $x$ is in the complement of $[-1/2, 1/2]$,  the functions $f,g,h, e_3, \ldots, e_{2N}$ coincide with $F, G, H, E_3, \ldots, E_{2N}$ respectively.

Finally we will show that $v$ is $\alpha$-concave in the set $$S_{\ve}= \{ (x,y) \ | \  -1\le x \le 1, \ g(x) < y < g(x)+ \ve \},$$
for $\ve>0$ sufficiently small.  

For this, we need to show that for all $x\in [-1 ,1]$ and $y-g(x)$ sufficiently small and positive we have
$$vv_{xx} - (1-\alpha) v_x^2<0, \ vv_{yy} - (1-\alpha) v_y^2<0$$
and
$$ (v v_{xx} - (1-\alpha) v_x^2)(v v_{yy} - (1-\alpha) v_y^2) - (vv_{xy} - (1-\alpha)v_x v_y)^2 >0.$$
We compute the terms above separately using (\ref{vcomp}),
\begin{equation} \label{vvxx}
\begin{split}
\lefteqn{v v_{xx} - (1-\alpha) v_x^2} \\
= {} &  \bigg( (y-g)f+ (y-g)^2h \bigg) \bigg( -g''f - 2g'f' + 2(g')^2 h + (y-g)f''   \bigg)  \\{}&  - (1-\alpha) \bigg(-g'f + (y-g)f' - 2(y-g)g' h + (y-g)^2 h'\bigg)^2 \\{}& + (y-g)^2 O(|g'|, |g''|)+ O(|y-g|^3) \\
= {} &  - (1-\alpha) (g')^2 f^2 + (y-g) (-g''f^2 - 2\alpha f f'g' - 2(1-2\alpha) f (g')^2 h)  \\ {} & + (y-g)^2 \bigg( ff'' - (1-\alpha) (f')^2 + O(|g'|, |g''|)   \bigg), \\
{} & + O(|y-g|^3).
\end{split}
\end{equation}
Observe that the zero order term in $(y-g)$ is negative, as is the first order term $-(y-g)g''f^2$.  To deal with the term $-2\alpha(y-g) f f' g'$ we argue as follows. 
If $\alpha\neq 0$, define $$\eta = \frac{1-4\alpha^2}{1-\alpha^2} \in (0,1),$$
and use the inequality
$$-2\alpha(y-g)  ff'g' \le (1-\eta)(1-\alpha) (g')^2 f^2 + (y-g)^2 \frac{\alpha^2}{(1-\alpha)(1-\eta)} (f')^2,$$
to obtain on $S_{\ve}$, for $\ve>0$ sufficiently small,
\[
\begin{split}
v v_{xx} - (1-\alpha) v_x^2
\le {} &  - \eta (1-\alpha) (g')^2 f^2 -2 (y-g) (1-2\alpha) f (g')^2 h  \\ {} & + (y-g)^2 \bigg( ff'' - \left\{ (1-\alpha) - \frac{\alpha^2}{(1-\alpha)(1-\eta)} \right\} (f')^2 + O(|g'|, |g''|)   \bigg) \\
{} & + O(|y-g|^3) \\
\le {} & - \frac{\eta}{2} (1-\alpha) (g')^2 f^2  + (y-g)^2 \bigg( ff'' - \frac{2}{3} (1-2\alpha) (f')^2 + O(|g'|, |g''|)   \bigg), \\
{} & + O(|y-g|^3), \\
\end{split}
\]
where have absorbed the first order term in the zero and second order terms,  using the fact that $f\ge c>0$ and $h$ is bounded.  This inequality holds for $\alpha=0$ too, taking $\eta=1$.
By definitions of $f$ and $g$, there exists a small constant $\rho>0$ such that if $-\delta-\rho \le x \le \delta+\rho$ then 
$$ ff'' - \frac{2}{3} (1-2\alpha) (f')^2 =-1$$
and 
$g'$, $g''$ is sufficiently small so that the second order term is strictly negative.  Otherwise $(g')^2$ is uniformly positive and so the zero order term dominates.  In either case there exists a uniform $a>0$ such that on $S_{\ve}$,
\begin{equation} \label{concavity1}
v v_{xx} - (1-\alpha) v_x^2 \le - a(y-g)^2,
\end{equation}
 as long as $\ve>0$ is sufficiently small.
 
 Next we compute on $S_{\ve}$,
\begin{equation} \label{vvyy}
\begin{split}
v v_{yy} - (1-\alpha) v_y^2   
= {} &  ((y-g)f + (y-g)^2 h)2h \\
&- (1-\alpha)( f  + 2(y-g)h )^2  + O(|y-g|^2) \\
= {} &  - (1-\alpha) f^2 -2 (y-g)(1-2\alpha) fh + O(|y-g|^2) \\
\leq &-a<0, 
\end{split}
\end{equation}
for a uniform $a>0$ uniform as long as $\ve>0$ is sufficiently small, using the lower bound on $f$ of (\ref{lbf}).
 
 Next we compute
\[
\begin{split}
\lefteqn{v v_{xy} - (1-\alpha) v_x v_y} \\
= {} & \bigg( (y-g)f + (y-g)^2 h \bigg) \bigg( f'    -2g'h  +2(y-g)h' \bigg) \\ 
{} & - (1-\alpha) \bigg(-g'f + (y-g)f' - 2(y-g)g' h + (y-g)^2 h' \bigg) \bigg(f + 2(y-g)h  \bigg)  \\
& + (y-g)^2 O(|g'|)+ O(|y-g|^3) \\
= {} &   (1-\alpha)g' f^2  + (y-g)(\alpha ff' + 2(1-2\alpha) fg'h) \\
 {} & + (y-g)^2 ( -(1-2\alpha) hf' +(1+\alpha) fh' +O(|g'|)) + O(|y-g|^3).
\end{split}
\]
Combining with (\ref{vvxx}) and (\ref{vvyy}), we finally obtain
\begin{equation} \label{det}
\begin{split}
\lefteqn{(v v_{xx} - (1-\alpha) v_x^2)(v v_{yy} - (1-\alpha) v_y^2) - (vv_{xy} - (1-\alpha)v_x v_y)^2  } \\= {} &   (1-\alpha) (y-g) g''f^4    \\ {} & +(y-g)^2 \bigg( -(1-\alpha) f^2 (ff'' - (1-\alpha)(f')^2) - \alpha^2 f^2 (f')^2 + O(|g'|, |g''|)  \bigg) \\ {} & + O(|y-g|^3) \\ 
\ge {} &  a|y-g|^2,
\end{split}
\end{equation}
for some uniform $a>0$ uniform as long as $\ve>0$ is sufficiently small.  To see this inequality we argue as follows.   Note that the coefficient of $(y-g)$ is nonnegative.  There exists a small $\rho>0$ such that if $-\delta-\rho \le x \le \delta+\rho$ then 
$ ff'' - \frac{2}{3} (1-2\alpha) (f')^2 =-1$
and 
$g'$, $g''$ is sufficiently small so that for these values of $x$,
\[
\begin{split}
\lefteqn{-(1-\alpha) f^2 (ff'' - (1-\alpha)(f')^2) - \alpha^2 f^2 (f')^2+ O(|g'|, |g''|)} \\  
= {} & -(1-\alpha) f^2 ( ff'' - \frac{1-2\alpha}{1-\alpha}(f')^2 )+ O(|g'|, |g''|) \\
\ge {}  &  - (1-\alpha) f^2 (ff'' - \frac{2}{3} (1-2\alpha) (f')^2)+ O(|g'|, |g''|) \\
={} &  (1-\alpha) f^2  + O(|g'|, |g''|) >a>0,  
\end{split}
\]
using the fact that $1/(1-\alpha)\ge 1 \ge 2/3$.   For $x$ not in this range, $g''$ is uniformly positive and so the first order term in $(y-g)$ dominates.  This establishes (\ref{det}).

Combining (\ref{concavity1}), (\ref{vvyy}) and (\ref{det}) we see that $v$ is $\alpha$-concave on $S_{\ve}$ for $\ve>0$ sufficiently small.  Moreover, $v(x,y)$ agrees with the function $U(x,y)$ for $x$ in the complement of $[-1/2,1/2]$ and hence we can extend $v$  (by simply setting equal to $U(x, y)$ for $(x, y)\notin[-1,1]\times [-1,1]$) to give  a smooth $\alpha$-concave function,  still referred to as $v$, on the set
$$\Omega_0^{\ve} = \{ p\in \Omega_0 \ | \ \textrm{dist}(p, \partial \Omega_0) < \ve \}.$$
Moreover $v$ has nonvanishing derivative on the boundary $\partial \Omega_0$.
 Applying Lemma \ref{Ghomi} completes the proof of the theorem.
\end{proof}

\section{Convexity breaking} \label{sectionlast}

In this section we complete the proof of Theorem \ref{maintheorem}.

 Let $k \ge 2$ be given, and choose $N=N(k)$ sufficiently large as in  Theorem \ref{thmste}.
We take our initial data $u_0$ to be the function $v \in C^{\infty}(\ov{\Omega}_0)$ constructed in Theorem \ref{thmconstruction} for this given $N$ so that by Theorem \ref{thmste} we have a nondegenerate $C^k$ solution $(\Omega,u)$ of the Stefan problem with this initial data on $[0,T]$.  Let $p$ be the point $(0, 1/20) \in \mathbb{R}^2$ and consider the disc $D_{\delta/2}(p)$ of radius $\delta/2$ centered at $p$.  
 Shrinking $T$ if necessary, the solution $u$ has nonvanishing derivative on the boundary $\partial \Omega_t$. Hence
 by the Implicit Function Theorem, the free boundary $\partial \Omega_t \cap D_{\delta/2}(p)$ is given by a graph $y= w(x,t)$
for a locally defined $C^k$ function $w$.  Moreover, $w(x,0) = 1/20$ and by the Stefan boundary condition (\ref{sbc}) the function $w_t(x,0)$ coincides with $-v_y(x, 1/20)$ and so
by part (e) of Theorem \ref{thmconstruction}, 
$$x\mapsto w_t(x, 0)$$ is strictly negative and strongly concave.  Here, and in what follows, we may increase $k$ without loss of generality as necessary.

To show that $\Omega_t$ is not convex it is sufficient to show that for $t \in (0,T]$ the function $w(x,t)$ is not convex as a function of $x$.  This however is an immediate consequence of Taylor's Theorem which gives
$$w(x,t) = 1/20 + w_t(x,0)t + R(x,t)t$$
for $x \in [-\delta/2, \delta/2]$ and $t \in [0,T]$, where 
$$R(x,t) = \int_0^t w_{tt}(x,s) \frac{(t-s)}{t}ds.$$
As $t \rightarrow 0$, the remainder term $R(x,t)$ tends to  zero   in the $C^2$ norm with respect to $x$.  Hence, shrinking $T$ if necessary we obtain $w_{xx}(x,t) <0$ for $t\in (0,T]$ and in particular, $x\mapsto w(x,t)$ is not convex.

This  immediately implies that $u|_{\Omega_t}$ is not $\alpha$-concave  for $t\in (0,T]$.  Indeed, fix $t \in (0,T]$, write $u=u|_{\Omega_t}$ and consider the set $\{ u(x,y) = \ve \}$ for a small $\ve>0$, near the point $(0, 1/20)$.  Shrinking $\ve$ if necessary we may assume that $\nabla u$ does not vanish there and $\{ u= \ve \}$ is given locally by a $C^2$ graph $y=\rho(x)$ which has $\rho''(0)<0$.  Rotating and translating the coordinates we may assume that $\rho(0)=\rho'(0)=0$ and hence $u_x(0, 0)=0$.  Differentiating twice the equation $u(x, \rho(x))=\ve$ gives $u_{xx}(0, 0) = - u_y(0, 0) \rho''(0)>0$.  Hence $u u_{xx} - (1-\alpha) u_x^2>0$ at $(0,0)$ and so  $u$ is not $\alpha$-concave.

This completes the proof of Theorem \ref{maintheorem}.

\end{document}